\documentclass[11pt,a4paper]{article}
\usepackage[T1]{fontenc}
\usepackage{lmodern}
\usepackage{amsmath,amssymb,amsthm,mathtools}
\usepackage{array,longtable,booktabs,tabularx}
\usepackage[margin=25mm,headheight=15pt]{geometry}
\usepackage{mathrsfs,graphicx,xcolor,enumitem}
\usepackage[expansion=false]{microtype}
\usepackage{fancyhdr}
\usepackage[unicode,hypertexnames=false,colorlinks=true,linkcolor=blue!40!black,
  citecolor=blue!40!black,urlcolor=blue!40!black]{hyperref}
\hypersetup{pdftitle={Local rigidity of minimal hypersurfaces},
  pdfsubject={Local rigidity in the four-sphere by compactness of finite jets}}
\allowdisplaybreaks[2]
\setlist{topsep=4pt,itemsep=2pt}
\numberwithin{equation}{section}
\newtheorem{theorem}{Theorem}[section]
\newtheorem{lemma}[theorem]{Lemma}
\newtheorem{proposition}[theorem]{Proposition}
\newtheorem{corollary}[theorem]{Corollary}
\theoremstyle{definition}

\theoremstyle{remark}
\newtheorem{remark}[theorem]{Remark}
\newcommand{\Scal}{\operatorname{Scal}}
\DeclareMathOperator{\tr}{tr}
\DeclareMathOperator{\Sym}{Sym}
\DeclareMathOperator{\Id}{Id}

\title{\textbf{On the Bryant conjecture for minimal hypersurfaces in $S^4$}\\[3pt]}

\date{}
\author{%
	Chao Qian\textsuperscript{1}
	\and
	Ya Tao\textsuperscript{2,*}%
}
\date{}
\begin{document}
\maketitle

\begingroup
\makeatletter

\renewcommand{\@makefntext}[1]{\noindent#1}
\makeatother
\renewcommand{\thefootnote}{}
\footnotetext{%
  \begin{minipage}[t]{\textwidth}
  \footnotesize\raggedright
  \setlength{\parindent}{0pt}%
  \setlength{\parskip}{0pt}%
  Chao Qian is supported by NSFC
  (No.~12371048) and Ya Tao is supported by Nankai Zhide Foundation.\par
\par
  \smallskip
  \noindent\rule{0.3\textwidth}{0.4pt}\par
  \smallskip
  \textsuperscript{*}\,(Corresponding author)\par
  \textit{E-mail:} \texttt{tao-ya@mail.nankai.edu.cn}\par
  \smallskip

  \textit{E-mail:} \texttt{6120150035@bit.edu.cn}\par
  
  \medskip
  \textsuperscript{1}\,School of Mathematics and Statistics, Beijing Institute of Technology,Beijing, 100081, P.R. China.\par
  
  \smallskip
  \textsuperscript{2}\,Chern Institute of Mathematics and LPMC, Nankai University, Tianjin 300071, P. R. China.
  \end{minipage}%
}
\endgroup

\begin{abstract}
We prove that a local minimal hypersurface $M^3$ in the unit sphere $S^4$
with constant scalar curvature and three distinct principal curvatures at every point
must be locally congruent to the minimal Cartan isoparametric hypersurface in $S^4$, which provides a positive answer to the Bryant conjecture  \cite{ChangLocal}. The key ingredient is to prove that, for each fixed \(S_0>3\), the set $\mathscr{G}_{S_0}$ of fourth-order jets of normalized local graph solutions satisfying \(|A|^2\equiv S_0\) is compact, where \(A\) denotes the shape operator. 
\end{abstract}
\medskip


\section{Introduction}\label{sec:strategy}

For a minimal hypersurface $M^3$ in the unit sphere $S^4$, let $A$
denote the shape operator with respect to a local unit normal.
The Gauss equation gives
$\operatorname{Scal}=6-|A|^2$, and hence constant scalar curvature is
equivalent to constancy of $|A|^2$. The study of such hypersurfaces is closely
related to the classical rigidity results of Chern, do Carmo and
Kobayashi~\cite{CDK} concerning the length of the second fundamental
form. For a minimal hypersurface $M^3$ in $S^4$ with constant scalar curvature, Chang~\cite{ChangLocal} proved that if $M^3$ has multiple
principal curvatures at some point, then $|A|^2\in\{0,3\}$, without any
compactness assumption. Combined with the results of Peng and
Terng~\cite{PengTerng,PengTer} for closed minimal hypersurfaces with
three distinct principal curvatures everywhere, this yields the
classification of closed minimal hypersurfaces in $S^4$ with constant
scalar curvature~\cite{ChangLocal}. In the constant mean curvature
setting, de Almeida and Brito~\cite{AB} proved that a closed hypersurface
of $S^4$ with constant mean curvature and constant nonnegative scalar
curvature is isoparametric. Chang~\cite{ChangCMC} subsequently
established the same conclusion without the nonnegativity assumption.

Regarding the classification of local minimal hypersurface with constant scalar curvature in $S^4$, Chang \cite{ChangLocal} stated a conjecture of Robert Bryant as follows:

\textbf{Bryant Conjecture.} A piece of minimally immersed hypersurfaces of constant scalar curvature in $S^4$
is isoparametric.

In general, there has been no local version
of the Peng-Terng theorem \cite{PengTerng}.
The remaining local problem with three distinct principal
curvatures requires an additional argument, since the function $f=\tr(A^3)$
on a local hypersurface needs not to attain an extremum. In dimension $n>3$, de Almeida, Brito,
Scherfner and Weiss~\cite[Theorem~3.1]{ABSW} proved that a local hypersurface
in $S^{n+1}$ with constant mean and scalar curvatures and exactly three
distinct principal curvatures everywhere is isoparametric. 


In this paper, we study the local rigidity of minimal hypersurfaces in $S^4$ with
constant scalar curvature and three distinct principal curvatures.
\begin{theorem}\label{thm:main}
Let $M^3$ be a local minimal hypersurface of the unit sphere $S^4$ with three distinct principal curvatures at every point and constant scalar curvature. Then $M^3$ must
have constant principal curvatures. Moreover, $M^3$ is locally congruent to the minimal Cartan isoparametric hypersurface in $S^4$.
\end{theorem}
\begin{remark}
Throughout the paper, a local hypersurface means a connected open piece of a smoothly immersed hypersurface. No completeness or compactness assumption is imposed. Shape operators and principal curvatures are understood with respect to a locally chosen unit normal.
\end{remark}
Together with Chang's result \cite{ChangLocal}, Theorem \ref{thm:main} completely solves the Bryant conjecture.

\begin{corollary}\label{cor:full-local}
Let $M^3$ be a local minimal hypersurface of $S^4$
with constant scalar curvature. Then $S=|A|^2\in \{0,3,6\}$ and the
principal curvatures are constant. Moreover, $M^3$ is locally congruent to a minimal isoparametric hypersurface in $S^4$.
\end{corollary}

The paper is organized as follows. In Section \ref{sec2}, we introduce the basic notations and geometric identities, and establish three key lemmas needed for the proof of Theorem \ref{thm:main}. In Section \ref{sec3}, we formulate the geometric conditions as an analytic system for normalized graph jets and prove compactness of the space of fourth-order jets of normalized local graph solutions with fixed $S_0>3$. Finally, Section \ref{sec:finish} combines this compactness result with the rigidity lemma at critical points of the cubic invariant $\tr (A^3)$ to prove Theorem \ref{thm:main}.

\section{Basic notations and three crucial lemmas}\label{sec2}
\subsection{Tensor convention}
Let $M^3$ be a hypersurface in $S^4$. Choose a local unit normal vector field $\nu$, and let $A$ denote the shape operator with respect to $\nu$. Let $g, \nabla$ and $h$ denote the induced metric, the Levi-Civita connection  and the second fundamental form of $M^3$ respectively. 
We use the induced metric $g$ to identify $\nabla^mA$
with the corresponding fully covariant tensors $\nabla^m h$.
More precisely, the $(1,1)$-tensor $A$ is identified with the
$(0,2)$-tensor $h$ by
\begin{equation*}\label{con1}
	h(X,Y)=g(AX,Y).
\end{equation*}
Since $\nabla g=0$, lowering the contravariant index commutes with
covariant differentiation. In general, after lowering the contravariant index of $\nabla^m A$, one can obtain
$\nabla^m h$. Accordingly, we can regard $\nabla^m A$ as a
covariant $(m+2)$-tensor $\nabla^m h$.
Choose a local orthonormal frame $\{e_1,e_2,e_3\}$ of $TM$.
For instance, we use the following tensor components convention:
\[(\nabla A)_{ijk}=(\nabla A)^i_{jk}=(\nabla h)_{ijk}= h_{ij;k}\quad \text{and}\quad (\nabla ^2A)_{ijkl}=(\nabla ^2A)^i_{jkl}=(\nabla ^2h)_{ijkl}=h_{ij;kl}.\]

\subsection{Geometric preliminaries and reduction}
Let $M^3$ be a minimal hypersurface with constant scalar curvature in the unit sphere $S^4$. Choose a unit normal vector field and let $A$ be the shape
operator of $M^3\subset S^4$. Moreover, we assume $M^3$ has three distinct principal curvatures.  Set
\begin{equation}\label{human3}
 S:=|A|^2=\tr(A^2)\quad \text{and} \quad f:=\tr(A^3).
\end{equation}
Minimality and the Gauss equation imply
\begin{equation}\label{humantrA}
	 \tr A=0\quad \text{and} \quad  \Scal=6-S,
\end{equation}
 where Scal is the scalar curvature of $M^3$. Thus $S$ is constant. Let  $\lambda_1,\lambda_2,\lambda_3$ be the principal
curvature functions with principal orthonormal frame $e_1,e_2,e_3$. Then
\begin{equation}\label{eq:main-cubic}
	f=\sum_i\lambda_i^3=3\lambda_1\lambda_2\lambda_3\quad \text{and} \quad
 \det(t\Id-A)=t^3-\frac S2t-\frac{f}{3},
\end{equation}
for $t\in \mathbb{R}$. Let $$T=\nabla A,\quad
U=\nabla ^2A \quad \text{and} \quad V=\nabla^3 A.$$
 Differentiating \eqref{human3}, we get
\begin{equation}\label{human4}
	\sum_{i=1}^{3}\lambda_iT_{iik}=0\quad \text{and} \quad\sum_{i=1}^{3}3\lambda_i^2T_{iik}=e_k(f), \quad\forall k=1,2,3.
\end{equation}
Twice differentiating \eqref{humantrA} gives
\begin{equation}\label{humantrA2}
\sum_{i=1}^{3}T_{iik}=0, \quad
\sum_{i=1}^{3}U_{iikl}=0\quad\text{for}\quad k,l=1,2,3.
\end{equation}

Simons' identity \cite{Simons} gives
\begin{equation}\label{found:Simons}
 \Delta A=(3-S)A,\quad C:=|\nabla A|^2=S(S-3),
\end{equation}
where $\Delta$ denotes the rought laplacian (See Section 1.2 of \cite{Simons}).
From (\ref{found:Simons}) we know that $S=0$ or $S\geq3$. The first case is incompatible with distinct
principal curvatures. If $S=3$, then $\nabla A=0$. Consequently, Ricci identity (See \eqref{human1}) implies
\[
 (\lambda_i-\lambda_j)(1+\lambda_i\lambda_j)=0
 \quad(i\ne j).
\]
Distinct eigenvalues would force every pairwise product to be $-1$,
which is impossible for three distinct principal curvatures. Hence $S>3$.

\subsection{Two Bochner-type identities}\label{app:identities}
In this subsection, let $M^3$ be a minimal hypersurface with constant scalar curvature in $S^4$. However, we do not assume $M^3$ has three distinct principal curvatures. Choose a local orthonormal frame $\{e_1,e_2,e_3\},$ and write
\[
h_{ij}=\langle Ae_i,e_j\rangle,\quad
T_{ijk}=h_{ij;k},\quad U_{ijkl}=h_{ij;kl}.
\]
A semicolon lists covariant derivatives in the order in which they are taken: $h_{ij;kl}=(\nabla_{e_l}\nabla h)_{ijk}$, where $h$ is the second fundamental form.  Moreover, it follows from the Codazzi equation that $T$ is completely symmetric.
 The Gauss equations are
\begin{align}\label{found:Gauss}
	R_{klib}
	=\delta_{li}\delta_{kb}-\delta_{ki}\delta_{lb}
	+h_{li}h_{kb}-h_{ki}h_{lb}.
\end{align}

For any covariant $r$-tensor $K$ with slots $I=(i_1,\ldots,i_r)$, our Ricci indentity's convention is
\begin{equation}
	K_{I;kl}-K_{I;lk}
	=\sum_{s=1}^r\sum_{b=1}^3 R_{kl i_sb}K_{I(s\mapsto b)}.
	\label{found:comm}
\end{equation}
Here $I(s\mapsto b)$ means replacement of the $s$th slot by $b$. 	At a point $p$ where $A$ is diagonal and $A(e_i) =\lambda_ie_i$, Codazzi and \eqref{found:comm} give
\begin{equation}\label{human1}
	U_{ijkk}-U_{kkij}
	=\delta_{ij}(\lambda_i-\lambda_k)(1+\lambda_i\lambda_k).
\end{equation}

Differentiating \eqref{found:comm} and taking appropriate contractions give the following formula
	\begin{align}
	\bigl(\Delta\nabla K-\nabla\Delta K\bigr)_{Ic}
	={}&\sum_{a,b,s}
	\left[(\nabla_aR_{ca i_sb})K_{I(s\mapsto b)}
	+2R_{ca i_sb}K_{I(s\mapsto b);a}\right]\notag\\
	&+\sum_{a,b}R_{ca ab}K_{I;b},
	\label{found:lapcomm}
\end{align}
where $1\leq a,b,c\leq 3, 1\leq s\leq r$.

Since the scalar curvature of $M^3$ is constant, it follows that $0=\frac12\Delta S=|T|^2+\langle h,\Delta h\rangle$.
Applying \eqref{found:lapcomm} to $K=h$, differentiating the first equation of \eqref{found:Simons}, and using
\begin{equation}\label{humanxie}
\nabla_{m}R_{kl ib}
=T_{lim}h_{kb}+h_{li}T_{kbm}
-T_{kim}h_{lb}-h_{ki}T_{lbm},
\end{equation}
we obtain, at a point where $A$ is diagonal,
\begin{equation}
	(\Delta T)_{ijk}
	=\left[9-S-\lambda_i^2-\lambda_j^2-\lambda_k^2
	+2(\lambda_i\lambda_j+\lambda_i\lambda_k+\lambda_j\lambda_k)\right]T_{ijk}.
	\label{found:lapT}
\end{equation}

Define the scalar contractions
\[
\mathcal A=\sum_{ijk}\lambda_i^2T_{ijk}^2,\quad
\mathcal B=\sum_{i,j,k}\lambda_i\lambda_jT_{ijk}^2,\quad
E=\sum_{i,j,k}\lambda_iT_{ijk}^2,\quad
\mathcal D=\sum_{i,j,k}\lambda_i^2\lambda_jT_{ijk}^2.
\]
By \eqref{found:lapT},  considering $\langle T, \Delta T\rangle$, we get the ordinary Bochner-type identity
\begin{equation}
	0=\tfrac12\Delta C
	=|U|^2+(9-S)C-3\mathcal A+6\mathcal B.
	\label{found:ordinary}
\end{equation}

For a trace-free three-by-three matrix, by Cayley--Hamilton theorem, we obtain
\begin{equation}\label{eq:CH}
	A^3=\frac S2A+\frac f3I,\quad
	\tr A^4=\frac12(\tr A^2)^2,\quad
	\tr A^5=\frac56\tr A^2\tr A^3.
\end{equation}
Let $B$ be symmetric with $\tr B=\tr(AB)=0$. Comparing the quadratic
terms after replacing $A$ by $A+sB$ in the last two identities gives
	\begin{align}
		2\tr(A^2B^2)+\tr(ABAB)&=\frac S2\tr(B^2),\label{eq:1}\\
		\tr(A^3B^2)+\tr(A^2BAB)
		&=\frac f6\tr(B^2)+\frac S2\tr(AB^2).\label{eq:2}
	\end{align}
By taking
$B=\nabla_kA$ in \eqref{eq:1} and summing over $k$ we get \begin{equation}
	2\mathcal A+\mathcal B=\tfrac12SC.
	\label{found:contractions1}
\end{equation} Multiplying $B^2$ to \eqref{eq:CH} and taking trace, one can obtain
\begin{equation}\label{eq:4}
	\tr (A^3B^2)=\frac S2\tr( AB^2)+\frac f3\tr( B^2).
\end{equation}
It follows from \eqref{eq:2} and \eqref{eq:4} that
$\tr(A^2BAB)=-f\tr(B^2)/6$, and then 
\begin{equation}
	\mathcal D=-\tfrac16 fC.
	\label{found:contractions}
\end{equation}
By using \eqref{eq:4} again, 
\[
G:=\sum_{i,j,k}\lambda_i^3T_{ijk}^2=\frac S2E+\frac f3C.
\]
In particular, by \eqref{found:contractions1}, \eqref{found:ordinary} becomes
\begin{equation}
	|U|^2=15\mathcal A-(2S+9)C.
	\label{found:simplified}
\end{equation}

For any covariant $r$-tensor $K$, define 
\[
L_AK:=\sum_{a,b}h_{ab}K_{;ab},\quad
Q_{ij}:=\sum_{a,b}T_{abi}T_{abj}.
\]
By taking the second derivative of $S$, it implies
\begin{equation}\label{human2}
	\sum_k\lambda_kU_{kkij}=-Q_{ij}.
	\end{equation} In light of \eqref{human2}, multiplying \eqref{human1} by $\lambda_k$ and summing over $k$ give
\begin{equation}
	(L_Ah)_{ij}=-Q_{ij}+\delta_{ij}(S\lambda_i^2-f\lambda_i-S),
\notag
\end{equation}
which is equivaluent to the tensor identity
\begin{equation}\label{found:weightedtensor}
	(L_Ah)_{ij}
	=
	-Q_{ij}
	+S\sum_a h_{ia}h_{aj}
	-fh_{ij}
	-S\delta_{ij}.
\end{equation}
Set
\[
\mathcal P:=\sum_{i,j,k,a,b}T_{ijk}T_{abk}U_{ijab},\quad
\mathcal W:=\sum_{i,j,k,l}\lambda_lU_{ijkl}^2.
\]
Differentiating \eqref{found:weightedtensor} and contracting with $T$ provide
\begin{equation}\label{eq:diff-LAh}
	\sum_{i,j,k}T_{ijk}\nabla_k(L_Ah)_{ij}
	=-2\mathcal P+2SE-fC,
\end{equation}
where we use the fact that the covariant $3$-tensor $df\otimes h$ is orthogonal to $T$, 
and $T$ is fully symmetric tensor.
The product rule for the coefficient $h_{ab}$ and \eqref{found:lapcomm} give
\begin{equation}\label{eq:LA-commutator}
(L_A\nabla h)_{ijk}-\nabla_k(L_Ah)_{ij}
=-\sum_{a,b}T_{abk}U_{ijab}+\mathcal R_{ijk},
\end{equation}
where 
\begin{align*}
	\mathcal R_{ijk}=\sum_{a,b}\lambda_a\bigl[
	&(\nabla_aR_{ka ib})h_{bj}
	+(\nabla_aR_{ka jb})h_{ib}
	+2R_{ka ib}T_{bja}+2R_{ka jb}T_{iba}
	+R_{ka ab}T_{ijb}\bigr]
\end{align*}
at a point where $A$ is diagonal.
Substitution \eqref{found:Gauss} and its covariant derivatives into $\mathcal R_{ijk}$ implies
\begin{equation}\label{eq:curvature-contraction}
	\sum T_{ijk}\mathcal R_{ijk}
	=(S+3)E+6\mathcal D-3G
	=\left(3-\frac S2\right)E+6\mathcal D-fC.
\end{equation}
By \eqref{eq:diff-LAh}, \eqref{eq:LA-commutator}, and \eqref{eq:curvature-contraction},
it follows from $\frac12L_AC=0$ that
\[
0=\mathcal W-3\mathcal P+\left(3+\frac{3S}{2}\right)E+6\mathcal D-2fC.
\]
Moreover, using \eqref{found:contractions}, we obtain the weighted Bochner-type identity
\begin{equation}
	\mathcal W-3\mathcal P+\frac32(S+2)E-3fC=0.
	\label{found:weighted}
\end{equation}

\subsection{An identity at a critical point of $f$}
Let $M^3\subset S^4$ be a local minimal hypersurface with three distinct principal curvatures and constant scalar curvature. At a critical point $p$ of $f=\tr (A^3)$, after fixing a principal orthonormal basis of $T_pM$, by using  \eqref{human4} and the first equation of \eqref{humantrA2}, we have $T_{ijk}=t\tau_{ijk}$ for $t\in \mathbb{R}$, where $\tau_{ijk}=1$ when the three indices are distinct and is zero otherwise, and $t^2=C/6$. Let $I=(i,j,k)$, and define
\[
\alpha_I=9-S-\sum_{s=1}^3\lambda_{I_s}^2
+2\sum_{s<r}\lambda_{I_s}\lambda_{I_r},
\]
where $I_s$ denotes the $s$-th index of $I$. \eqref{humanxie} gives that
\begin{align*}
	\nabla_aR_{kl ib}
	&=t\left(\tau_{lia}\lambda_k\delta_{kb}
	+\tau_{kba}\lambda_l\delta_{li}
	-\tau_{kia}\lambda_l\delta_{lb}
	-\tau_{lba}\lambda_k\delta_{ki}\right).
\end{align*}
Differentiating the tensor form of \eqref{found:lapT} and applying \eqref{found:lapcomm} to $K=T$ gives
\begin{equation}\label{found:criticalDeltaU}
	\begin{aligned}
		(\Delta U)_{Ic}=&\nabla_c(\Delta T)_I+\left(\Delta\nabla T-\nabla\Delta T\right)_{Ic}\\
		=&\alpha_I U_{Ic}
		+t^2\sum_{s,b}
		\left(-\lambda_{I_s}-\lambda_b
		+2\sum_{r\ne s}\lambda_{I_r}\right)
		\tau_{I_sbc}\tau_{I(s\mapsto b)}\\
		&+\sum_{a,b,s}\left[
		t(\nabla_aR_{caI_sb})\tau_{I(s\mapsto b)}
		+2 R_{caI_sb}U_{I(s\mapsto b),a}\right]
		+\sum_{a,b} R_{caab}U_{Ib},
	\end{aligned}
\end{equation}
where all sums run over $\{1,2,3\}$. This equation involves only $\lambda_i$, $t^2$, and $U$, and it determines the contraction $\langle U,\Delta U\rangle$, which will be used in the proof of Lemma \ref{lem:critical}.

\subsection{Three crucial lemmas}\label{sec:inputs}

By definition, a normalized tensor triple means all the components of
$
(A,T,U)=(A,\nabla A,\nabla^2A)
$
at a chosen point $p$, after fixing a unit normal and an orthonormal tangent
basis of $T_pM$. Equivalently, a normalized tensor triple can be regarded as an element of  
$$\Sym^2(\mathbb{R}^3)^*\oplus\Sym^3(\mathbb{R}^3)^*\oplus(\Sym^3(\mathbb{R}^3)^*\otimes (\mathbb{R}^3)^*),$$where $\Sym^k(\mathbb{R}^3)^*$ denotes the space of symmetric covariant
$k$-tensors on $\mathbb{R}^3$.
A limit \((A_\infty,T_\infty,U_\infty)\) of normalized tensor triples is a componentwise limit
\[
(A_n,T_n,U_n)|_{p_n}\longrightarrow(A_\infty,T_\infty,U_\infty),
\]
where $A_n,T_n,U_n$ are obtained from a sequence of local pointed minimal hypersurfaces $(M_n^3, p_n)$. Similarly, we can also define the normalized quadruple $
(A,T,U,V)=(A,\nabla A,\nabla^2A,\nabla^3A)
$, and also the limit of normalized tensor quadruples.

The next lemma shows that the property of three distinct principal curvatures is preserved under limits of normalized tensor triples. 
\begin{lemma}\label{lem:repeated}
For fixed $S_0>3$ and a family of local minimal hypersurfaces $M_n^3\subset S^4$ with constant $|A_n|^2=S_0$, choose points \(p_n\in M_n\).  If  the sequence of normalized triples $(A_n,T_n,U_n)|_{p_n}$ converges to a limit $(A_\infty,T_\infty,U_\infty)$, then $A_\infty$ has three distinct eigenvalues.
\end{lemma}


\begin{proof}
	The proof is essentially due to Chang~\cite{ChangLocal}. 
	
	First, let $M^3$ be a minimal hypersurface in $S^4$ with constant  $|A|^2=S_0>3$. We will show that $M^3$ has three distinct principal curvatures at every point. Assume there exists a point in $M^3$ such that $A=\operatorname{diag}(\lambda,\lambda,-2\lambda)$ with $S_0=6\lambda^2>3$ at this point by fixing a principal orthonormal basis. Reverse the unit normal of $M^3$ if necessary so that
	$\lambda>0$. The symmetry of $T$, the first
	equation of \eqref{human4} and \eqref{humantrA2} together imply, after a
	rotation of the multiple $\lambda$-eigenspace,
	\[
	T_{111}=-T_{122}=p,\quad T_{112}=-T_{222}=q,\quad
	T_{113}=-T_{223}=c,\quad T_{123}=0,
	\]
for some $p,q,c\in \mathbb{R}$.	Any other component follows from the symmetry of $T$ or vanishes. Define
	\[
	a:=c^2,\quad b:=p^2+q^2.
	\]
	By \eqref{found:Simons}, $C=6a+4b$. Write
	\[
	x=U_{1111}=U_{2222},\quad y=U_{1112},\quad
	Z=\lambda^2\bigl(3U_{1123}^2+3U_{2213}^2+U_{1113}^2+U_{2223}^2\bigr).
	\]
	The equality defining $x$ follows from the trace equations \eqref{humantrA2},\eqref{human2}  and Ricci
	identity \eqref{human1}, which also imply 
	\[
	\begin{gathered}
		U_{3311}=U_{3322}=\frac{2(a+b)}{3\lambda},\quad
		U_{1133}=U_{2233}=-\frac{a}{3\lambda},\quad
		U_{3333}=\frac{2a}{3\lambda},\quad U_{3312}=0,\\
		U_{3313}=\frac{2pc}{3\lambda},\quad
		U_{3323}=\frac{2qc}{3\lambda},\quad
		U_{1122}=-x-\frac{2(a+b)}{3\lambda},\quad U_{1222}=-y.
	\end{gathered}
	\]
	The remaining components satisfy the trace relations and are symmetric
	under all permutations, except those with index multiset $\{1,1,3,3\}$
	or $\{2,2,3,3\}$, whose two values have just been displayed.
	Set
	\[
	N=|U|^2,\quad N_3=\sum_{i,j,k}U_{ijk3}^2,\quad
	\mathcal P=\sum_{i,j,k,l,m}T_{ijk}T_{klm}U_{ijlm}.
	\]
	Substituting the above component relations into the definitions of
	$N$, $N_3$, and $\mathcal P$, with the remaining unrestricted components in $Z$,
	gives
	\begin{align*}
		&\lambda^2N=4Z+8\lambda^2(x^2+y^2)+8\lambda(a+b)x
		+\frac{58a^2+112ab+48b^2}{9},\\
		&\lambda^2N_3=Z+\frac{10a^2+12ab}{9},\\
		&\lambda\mathcal P=4a\lambda x+\frac{8a^2-4b^2}{3}.
	\end{align*}
	Using the ordinary and weighted Bochner-type identities
	\eqref{found:simplified} and \eqref{found:weighted}, we obtain
	\[
	N=15\mathcal A-(2S_0+9)C,\quad
	\mathcal W-3\mathcal P+\tfrac32(S_0+2)E-3fC=0.
	\]
	At the chosen point with multiple principal curvatures,
	\[
	\mathcal A=\lambda^2(12a+4b),\quad
	\mathcal W=\lambda N-3\lambda N_3,\quad E=4\lambda b,
	\quad f=-6\lambda^3.
	\]
	In particular the ordinary Bochner-type identity becomes $N=3(S_0-3)C-10S_0b$.
	For an explicit cancellation, set
	\[
	R_o=N-15\mathcal A+(2S_0+9)C,\quad
	R_w=\mathcal W-3\mathcal P+\tfrac32(S_0+2)E-3fC.
	\]
	The combination $-\lambda^2R_o/12+\lambda R_w/3$ cancels all terms
	containing $Z$. Substitution of $S_0=6\lambda^2$ and
	$S_0(S_0-3)=6a+4b$, followed by completing one square, gives
	\begin{equation}\label{rep:squares}
		\begin{split}
			0={}&2\left(\lambda x-\frac{a-b}{2}\right)^2+2\lambda^2y^2
			+\frac{29}{6}a^2+\frac{251}{18}ab+\frac{113}{18}b^2
			+3S_0a+\frac{13}{4}S_0b.
		\end{split}
	\end{equation}
	Every term is nonnegative. Thus $a=b=0$, contradicting
	$6a+4b=S_0(S_0-3)>0$. 
	
	Next, let
$(A_n,T_n,U_n)|_{p_n}\longrightarrow(A_\infty,T_\infty,U_\infty). $
	We claim that $A_\infty$ also has three distinct eigenvalues. Indeed, all tensor identities used above are finite algebraic identities in $A,T,U$ (with the curvature terms expressed algebraically in $A$ by the Gauss equation), and hence they pass componentwise to the limiting triple. Suppose that $A_\infty$ has a multiple eigenvalue. Since
	\[\operatorname{tr}A_\infty=0,\quad |A_\infty|^2=S_0>3,\]
	after one fixes an orthonormal basis of $\mathbb{R}^3$ and, if necessary, reverses of the normal, we may write
	\[ A_\infty=\operatorname{diag}(\lambda,\lambda,-2\lambda),
	 \quad \lambda>0,\quad 6\lambda^2=S_0.\]
	Applying an orthogonal rotation in the multiple $\lambda$-eigenspace, one changes $T_\infty$ in the normal form used above. Since the limiting triple satisfies the same trace, Ricci, ordinary Bochner-type, and weighted Bochner-type identities, the preceding algebraic calculation applies verbatim and yields \eqref{rep:squares}. Consequently $a=b=0$, whereas the limiting Simons' identity gives
	\[6a+4b=|T_\infty|^2=S_0(S_0-3)>0,\]
	and we get a contradiction. Thus $A_\infty$ has three distinct eigenvalues.
\end{proof}

The next lemma shows that the rigidity at a critical point of the function $f$ defined in \eqref{human3}.
\begin{lemma}\label{lem:critical}
Let $M^3\subset S^4$ be a local minimal hypersurface with three distinct principal curvatures at every point and constant $|A|^2=S>3$. If there exists a point $p$ satisfying
$df(p)=0$, then  $f(p)=0$ and $S=6$.
\end{lemma}
\begin{proof}
Unless otherwise stated, all tensor components below are evaluated
at $p$ in a principal orthonormal basis. Since $df(p)=0$, 
	the three trace equations \eqref{human4}, and the first equation of \eqref{humantrA2},  imply that $T_{iik}=0$ for $i,k=1,2,3$. By Codazzi equation, only the components with three distinct
	indices remain. Thus if we wirte $T_{123}=t$, then $C=|T|^2=6t^2$.
  By the equation $\nabla |T|^2=0$, the second trace equations \eqref{humantrA2},\eqref{human2}, and Ricci
identity \eqref{human1},  $U$ is determined up to some $d\in \mathbb{R}$:
	\begin{align}
		&a_i:=U_{iiii}=\frac{6-S}{6}\lambda_i-\frac f3
		+d\left(\lambda_i^2-\frac{2S}{3}\right),\label{crit:ai}\\
		&\frac{U_{iijj}+U_{jjii}}2
		=\frac f6+\frac{2S-3}{6}\lambda_k
		+d\left(\lambda_k^2-\frac S6\right),\label{crit:U}\\
		&U_{iijj}-U_{jjii}=(\lambda_i-\lambda_j)(1+\lambda_i\lambda_j).
		\label{crit:skew}
	\end{align}
	Here $i,j,k$ are distinct. Moreover, the preceding
	equations imply that all multiplicity patterns not appearing in
	\eqref{crit:ai}--\eqref{crit:skew} vanish. The ordinary and weighted
	Bochner-type identities \eqref{found:simplified} and \eqref{found:weighted} become
	\begin{align}
		B:={}&9S^2d^2+6(4S-3)fd-3f^2-S(7S-15)(S-6)=0,
		\label{crit:B}\\
		W:={}&18Sf d^2+[S^2(19S-33)-12f^2]d
		+f(43S^2-126S+18)=0.\label{crit:W}
	\end{align}
	These are the only two initial scalar equations needed.
	Since $A$ has three distinct principal curvatures, we have
	\begin{equation}
		0\le f^2<\frac{S^3}{6},\quad
		D:=\frac{S^2}{6}-\frac{f^2}{S}>0.
		\label{crit:disc}
	\end{equation}
	Indeed $\prod_{i<j}(\lambda_i-\lambda_j)^2=S^3/2-3f^2$.
	If $f=0$, equation \eqref{crit:W} gives $d=0$, and then
	\eqref{crit:B} gives $S=6$. 
	
	 Now assume $f\ne0$, and we will prove that $S>5$. Set
	\[
	h=\frac3S\in(0,1),\quad y=\frac{f^2}{S^3}\in(0,1/6),
	\quad x=-\frac{3Sd}{f}.
	\]
	We note that in \eqref{crit:W}, the coefficient $S^2(19S-33)-12f^2$ of $d$ is positive by \eqref{crit:disc}, and the coefficient $18Sd^2+43S^2-126S+18$ of $f$ also positive for $S>3$, which implies $fd<0$, hence $x>0$.
	After division by positive powers of $S$, \eqref{crit:B}--\eqref{crit:W} have the following forms:
	\begin{align}
		&y[x^2-(8-2h)x-3]=(7-5h)(1-2h),\label{crit:Bn}\\
		&6yx(x+2)-(19-11h)x+129-126h+6h^2=0.\label{crit:Wn}
	\end{align}
	Eliminating $y$ by cross-multiplication yields
	\begin{equation}
		\begin{split}
			0=\mathcal C(h,x)={}&(19-11h)x^3+(-323+366h-88h^2)x^2\\
			&+(891-1005h+180h^2-12h^3)x+387-378h+18h^2.
		\end{split}\label{crit:cubic}
	\end{equation}
	
	First, we claim that $x>5$. If $h\le1/2$, equation \eqref{crit:Bn} gives
	$x\ge(7+\sqrt{61})/2>5$. if $1/2<h<1$, suppose $0<x\le5$.
	The endpoint factorizations are
	\begin{align*}
		\mathcal C(1,x)&=(x-3)^2(8x+3),\\
		\mathcal C(1/2,x)&=\frac{27}{2}(5-x)
		[13+3(5-x)-(5-x)^2].
	\end{align*}
	Both are nonnegative on $[0,5]$. If $2\le x\le5$, then
	\[
	\mathcal C_{hh}(h,x)\le-176x^2+324x+36<0,
	\]
	so strict concavity in $h$ gives $\mathcal C(h,x)>0$.
	If $0<x\le2$, then
	\[
	\mathcal C_h(h,x)\le-11x^3+278x^2-681x-342<0,
	\]
	where the last inequality follows from $278x^2\le556x$.
	Hence $\mathcal C(h,x)>\mathcal C(1,x)\ge0$ in this case as well.
	Both conclusions contradict \eqref{crit:cubic}. The claim holds.
	
	Now suppose $h\ge3/5$. For $3/5\le h\le1$ and $x\ge5$,
	\begin{align*}
		&\mathcal C(h,5)=(2h-1)(858-656h-30h^2)\ge90,\\
		&\mathcal C_x(h,5)=-914+1830h-700h^2-12h^3\ge-71,\\
		&\mathcal C_{xx}(h,x)\ge-76+402h-176h^2\ge100.
	\end{align*}
	The first polynomial is concave in $h$, with endpoint values
	$2268/25$ and $172$, hence $\mathcal{C}(h,5)\geq 90$. The other two displayed polynomials increase
	on $[3/5,1]$. Thus, with $w=x-5\ge0$, Taylor's formula gives that for some $\theta\in (0,1)$,
	\[
	\mathcal C(h,5+w)=\mathcal C(h,5)+w\mathcal C_x(h,5)+\frac{1}{2}w^2\mathcal C_{xx}(h,5+\theta w)\ge90-71w+50w^2>0,
	\]
	This contradicts \eqref{crit:cubic}. Thus $h<\frac{3}{5}$ and then 
	\begin{equation} S>5.\label{crit:S5}\end{equation}
	
Next, we establish the geometric form of the differentiated identity.
	The ordinary Bochner-type identity \eqref{found:simplified} holds on a neighborhood of $p$.
	Taking its Laplacian gives the necessary identity
	\begin{equation}
		0=|V|^2+\langle U,\Delta U\rangle-\frac{15}{2}\Delta\mathcal A.
		\label{crit:next}
	\end{equation}
	A useful improvement is to express $\Delta\mathcal A$ geometrically.
	Near $p$, choose a principal orthonormal frame and put
	\[
	\nu_i:=\lambda_j-\lambda_k
	\]
	for cyclic $(i,j,k)$. Since the vectors $(1,1,1), (\lambda_1,\lambda_2,\lambda_3) $, and $(\nu_1,\nu_2,\nu_3)$ are mutually orthogonal, for each fixed $r$, from \eqref{human4} and \eqref{humantrA2} we know that there is a unique
	scalar $u_r$ such that
\begin{equation}\label{human5}
		T_{iir}=\nu_i u_r .
\end{equation}
 Write $t=T_{123}$. Then
\begin{equation}\label{humanC}
		C=6t^2+\sum_r(5S+6\lambda_r^2)u_r^2.
\end{equation}
Similarly, the $t^2$-terms in $\mathcal A-SC/3$ cancel, while the
	coefficient of $u_r^2$ is
	\[
	6\lambda_r^4-3S\lambda_r^2-\frac{2S^2}{3}
	=
	2f\lambda_r-\frac{2S^2}{3}
	=
	-2\nu_r^2\left(\lambda_r^2+\frac S6\right).
	\]
	Since $T_{rrr}=\nu_r u_r$, we obtain the identity near $p$,
	\[
	\frac{SC}{3}-\mathcal A
	=
	2\sum_r\left(\lambda_r^2+\frac S6\right)T_{rrr}^2.
	\]
	
	At the critical point $p$, the Vandermonde argument above gives
	$T_{rrr}=0$.  Taking the Laplacian of the preceding 
	identity near $p$ and evaluating at $p$ therefore yield
	\[
	-\Delta\mathcal A
	=
	4\sum_r\left(\lambda_r^2+\frac S6\right)
	|\nabla T_{rrr}|^2.
	\]
	Moreover, the vanishing of the remaining multiplicity patterns of $U$ gives
	$U_{rrrs}=0$ for $s\ne r$, while $U_{rrrr}=a_r$.  Hence
	\begin{equation}
		\Delta\mathcal A=-4\sum_i\left(\lambda_i^2+\frac S6\right)a_i^2.
		\label{crit:deltaA}
	\end{equation}
	In particular, \eqref{crit:next} becomes
	\begin{equation}
		0=|V|^2+\langle U,\Delta U\rangle
		+30\sum_i\left(\lambda_i^2+\frac S6\right)a_i^2.
		\label{crit:positive-next}
	\end{equation}
	Only identities valid near $p$ have been differentiated. The pointwise
	assumption $df(p)=0$ has never been differentiated.
	The identity still uses $V=\nabla^3A$ as necessary compatibility
	information. Its norm will be determined below from $A,T,U$ at $p$.
	
	Set $N_i=\sum_{a,b,c}U_{abci}^2=a_i^2+3\sum_{j\ne i}U_{jjii}^2$.
	The third-derivative norm is
	\begin{equation}
		\begin{split}
			|V|^2={}&\frac{10}{3C}\sum_iN_i^2
			-\frac{38}{3}\sum_i\lambda_i^2N_i+\frac{34}{3}\sum_iN_i
			+\frac C6(58S^2-212S+330).
		\end{split}\label{crit:Vnorm}
	\end{equation}
In fact, the identity $\nabla_{ii}|T|^2=0$ gives
	\[
	0=N_i+6t\,T_{123;ii}.\]
Hence	\[
		\frac{T_{123;ii}}{t}=q_i:=-\frac{N_i}{C}.
	\]
For fixed distinct $i,j,k$, the twenty ordered components of $V/t$
whose five indices consist of three copies of $i$ and one copy each of
$j$ and $k$ are determined by Ricci identity.  More precisely,
\[
\frac{V_{iiijk}}{t}
=
\frac{V_{iiikj}}{t}
=
q_i+4\lambda_i^2-5,
\]
while, up to permutations of the first three indices,
\[
\frac{V_{iijik}}{t}
=
q_i+2\lambda_i^2-2\lambda_i\lambda_k-3,
\quad
\frac{V_{iikij}}{t}
=
q_i+2\lambda_i^2-2\lambda_i\lambda_j-3,
\]
\[
\frac{V_{iijki}}{t}
=
\frac{V_{iikji}}{t}
=
q_i+2\lambda_i^2-1,
\quad
\frac{V_{ijkii}}{t}=q_i.
\]
The respective multiplicities are $2,3,3,6,6$.
	The off-diagonal equations $\nabla_{ij}|T|^2=0$ together with the
	third-order trace relations \eqref{human3}, \eqref{humantrA} eliminate the
	remaining multiplicity patterns. Therefore, after squaring the above
	components and summing over $i$,
	\[
	\frac{|V|^2}{t^2}
	=
	\sum_i\Bigl[
	20q_i^2+(76\lambda_i^2-68)q_i
	+92\lambda_i^4+(12S-212)\lambda_i^2+110
	\Bigr].
	\]
	Using
	\[
	\sum_i\lambda_i^2=S,\quad
	\sum_i\lambda_i^4=\frac{S^2}{2},
	\quad
	q_i=-\frac{N_i}{C},
	\quad
	t^2=\frac{C}{6},
	\]
proves \eqref{crit:Vnorm}.
	
Set
	\[
	\mu_i=\lambda_i^2-\frac S3-\frac fS\lambda_i,\quad
	\beta=\frac{4S^2+6S-18+12fd-12Sd^2}{9C}.
	\]
	The vectors $(1,1,1), (\lambda_1,\lambda_2,\lambda_3) $, and $(\mu_1,\mu_2,\mu_3) $ are orthogonal in $\mathbb{R}^3$ by \eqref{eq:main-cubic},
	with squared norms $3,S,D$, respectively. The Bochner-type identities \eqref{found:ordinary}, \eqref{found:weighted} imply
	\[
	\sum_iN_i=3C(S-3),\quad
	\sum_i\lambda_iN_i=C(Sd+4f),
	\]
	where for the second equation, one uses
	$\mathcal P=\sum_{i,j,k,a,b}T_{ijk}T_{abk}U_{ijab}=C(f+Sd)/3$,
	which is obtained directly from \eqref{crit:ai}--\eqref{crit:skew}.
	Observe that the coefficient of
	$\lambda_i^2$ in $N_i$ is $C\beta$. Consequently, by reducing powers of $\lambda_i$ with
	$\lambda_i^3=(S/2)\lambda_i+f/3$,
	\begin{equation}\label{crit:Ni-projection}
		\frac{N_i}{C}=S-3+\left(d+\frac{4f}{S}\right)\lambda_i+\beta\mu_i.
	\end{equation}
This is the orthogonal decomposition of $(N_1,N_2,N_3)$ in $\mathbb R^3$.
	Substituting it into \eqref{crit:Vnorm} gives
	\begin{equation}\label{crit:V-projection}
		\frac{|V|^2}{C}
		=7S^2-\frac{70}{3}S+43+\frac{10S}{3}d^2+14fd
		+\frac{8f^2}{3S}+\frac D3(10\beta^2-38\beta).
	\end{equation}
	The two remaining contractions in \eqref{crit:positive-next}
are
\begin{equation}\label{human23}
	\begin{aligned}
		\langle U,\Delta U\rangle={}&
		-5S^4+\frac{257}{9}S^3-\frac{137}{3}S^2+29S
		+\left(\frac{118}{9}S-\frac{38}{3}\right)f^2\\
		&+\left(-\frac{56}{9}S^2+\frac{340}{3}S-68\right)fd
		+\left(-\frac{22}{9}S^3+34S^2+\frac{56}{3}f^2\right)d^2,\\
		G={}&\frac{S^2(S-6)^2}{54}
		+\left(\frac{5S}{18}-\frac23\right)f^2
		+\frac{S(S+6)}9fd+\left(\frac{S^3}{9}+\frac{f^2}{3}\right)d^2,
	\end{aligned}
	\end{equation}
		where $G=\sum_i(\lambda_i^2+S/6)a_i^2$. For the first contraction, we use \eqref{found:criticalDeltaU}. For the second contraction, one simply
	squares \eqref{crit:ai}.
Write
	\[
	e=Sd,\quad y=\frac{f^2}{S^3},\quad
	\alpha=(50+186y)S+216,
	\]
	\[
	\gamma=(88+72y)S+822-\frac{612}{S},\quad
	\zeta=271S-420-\frac{108}{S}.
	\]
	Substituting \eqref{crit:V-projection}, \eqref{human23} into  \eqref{crit:positive-next} gives the sum of squares identity
	\begin{equation}\label{crit:main-sos}
			0=30CD(\beta-1)^2+S(14S-25)(S-6)^2
			+\alpha e^2+\gamma ef+\zeta f^2.
	\end{equation}
	Complete the last quadratic form as
	\[
	\alpha e^2+\gamma ef+\zeta f^2
	=\frac\gamma3\left(e+\frac32f\right)^2
	+\kappa e^2+\eta f^2,
	\]
	where
	\[
	\kappa=\left(\frac{62}{3}+162y\right)S-58+\frac{204}{S},
	\quad
	\eta=(205-54y)S-\frac{2073}{2}+\frac{351}{S}.
	\]
	For $S>5$ and $0\le y<1/6$, both $\gamma$ and $\kappa$ are positive,
	and
	\[
	2S\eta\ge (S-5)(392S-113)+137>0.
	\]
	Every term in \eqref{crit:main-sos} is therefore nonnegative, with
	strictly positive coefficients of $e^2$ and $f^2$. It follows that $f=e=0$, contradicting $f\ne0$. Thus $f=0$. Therefore,
	\eqref{crit:B} and \eqref{crit:W} then give $d=0$
	and $S=6$.
\end{proof}

The next lemma establishes the key nondegeneracy needed in the proof of Proposition \ref{spacedec}.
\begin{lemma}\label{lem:no-complex-characteristic}
Fix $S_0>3$. Let $M^3\subset S^4$ be a local minimal hypersurface with constant $|A|^2=S_0$. Given $p\in M^3$ and a principal orthonormal basis $e_1, e_2, e_3$ of $T_pM$, write $A(e_i)=\lambda_ie_i$.  At the point $p$, the components of $A$ and $T$ determine three homogeneous polynomials $q_0,q_1,q_2$ by
	\begin{align*}
		q_0(\xi)=\sum_i\xi_i^2,\quad
		q_1(\xi)=\sum_i\lambda_i\xi_i^2,\quad
		q_2(\xi)=\sum_{i,j,k}T_{ijk}\xi_i\xi_j\xi_k,\quad \xi=(\xi_1,\xi_2,\xi_3)\in \mathbb{C}^3.
	\end{align*}
	Then $q_0(\xi)=q_1(\xi)=q_2(\xi)=0$ implies $\xi=0$. Moreover, the same conclusion holds for the polynomials \(q_0^\infty,q_1^\infty,q_2^\infty\) associated with the limiting tensors \(A_\infty,T_\infty\), whenever
	\[
	(A_n,T_n,U_n,V_n)|_{p_n}\longrightarrow
	(A_\infty,T_\infty,U_\infty,V_\infty)
	\]
	is a convergent sequence of normalized tensor quadruples arising from pointed local minimal hypersurfaces $(M^3_n, p_n)$ satisfying \(|A_n|^2\equiv S_0\).
\end{lemma}
\begin{proof}
	First, consider $p\in M^3$. By Lemma \ref{lem:repeated}, $M^3$ has three dinstinct principal curvatures at every point. According to \cite{Nomizu}, we can choose a principal orthonormal frame $\{e_1,e_2,e_3\}$ near $p$.
For cyclic $(i,j,k)$, set
$
 \nu_i:=\lambda_j-\lambda_k.$
 Similar to \eqref{human5} in the proof of Lemma \ref{lem:critical}, for each $r\in \{1,2,3\}$, there is a unique scalar $u_r$ such that
 \[
 T_{iir}=\nu_i u_r.
 \]
 From \eqref{humanC},
	\begin{equation}\label{eq:energy-for-symbol}
		\sum_i w_iu_i^2+6T_{123}^2=S_0(S_0-3),\quad w_i:=5S_0+6\lambda_i^2
	\end{equation}
		Let $\{\omega^1,\omega^2,\omega^3\}$ be the dual coframe of $\{e_1,e_2,e_3\}$ and set
$\alpha=\sum_i u_i\omega^i.$
	The eigenvalue derivative formula gives intrinsically
	\[
	d\lambda_i=\nu_i\alpha,\quad
	d\nu_i=-3\lambda_i\alpha,	\quad0=d^2\lambda_i
	=d\nu_i\wedge\alpha+\nu_i\,d\alpha.
	=\nu_i\,d\alpha
	\] 
Thus 
$
	d\alpha=0.
$
	Define
	\[
	Q_{ij}:=(\nabla_{e_i}\alpha)(e_j),
	\quad
	T_i:=e_i(T_{123}).
	\]
	Then,
$
	Q_{ij}=Q_{ji}.
$

	Relabel the eigenvalues in the order \(\lambda_1>\lambda_3>\lambda_2\), and then the principal curvatures can be writen as follows:
	\begin{equation}\label{eq:symbol-param}
		\lambda=(\lambda_1,\lambda_2,\lambda_3)=R\left(\frac{t+2}{3},-\frac{2t+1}{3},\frac{t-1}{3}\right)
		\quad \text{for} \quad R>0,\quad t>0.
	\end{equation}
 Then we have \begin{equation}\label{humanlem1}
S_0=\frac{2D}{3}R^2\quad \text{for}\quad D:=t^2+t+1.
	\end{equation}
	Also let
	\[ P:=2t^2+3t+3,\quad Q:=3t^2+3t+2,\quad K:=2t^2+t+2. 
	\]
	The common projective zeros of \(q_0,q_1\) are represented, up to conjugation,
	by \[\xi=(\sqrt{-1}\varepsilon \sqrt t,\sqrt{-1}, \sqrt{t+1}),\] where \(\varepsilon\in \{\pm1\}\).
	Indeed \((\xi_1^2,\xi_2^2,\xi_3^2)\) is proportional to
	\((\nu_1,\nu_2,\nu_3)=R(-t,-1,t+1)\).
	
	We will argue by contradiction. Assume \(q_2(\xi)=0\) for \(\xi=(\sqrt{-1}\varepsilon \sqrt t,\sqrt{-1},\sqrt{t+1})\) and some $\epsilon=\pm 1$.
	For such a vector the equation \(q_2(\xi)=0\), separated into real and
	imaginary parts, is precisely
	\begin{equation}\label{eq:general-cubic-characteristic}
		u_1=:v,\quad u_2=-\varepsilon\frac{P}{Q}\sqrt t v,
		\quad u_3=:z,\quad T_{123}=\varepsilon\frac{KR}{3\sqrt t}z.
	\end{equation}
An explicit compatibility identity will be shown, which forces \(v=0\).
	Put
	\[
	\rho_i=-\frac{3\lambda_i(7S_0/2-3\lambda_i^2)}
	{\nu_i(3\lambda_i^2-S_0/2)}.
	\]
	We point out that the denominators of $\rho_i$ are nonzero. 
	Below we list some  formulas, for which specify the entire tensor \(U\) used in the
compatibility calculation:
\begin{align}
	&c_{ij\ell}:=\langle\nabla_{e_\ell}e_j,e_i\rangle
	=\frac{T_{ij\ell}}{\lambda_j-\lambda_i}\quad(i\ne j),
	\quad c_{ii\ell}=0,\label{humanlem2}\\
	&e_\ell (u_i)=Q_{\ell i}+\sum_s c_{si\ell}u_s,
	\quad e_\ell(\nu_i)=-3\lambda_i u_\ell,\label{humanlem3}\\\
	&e_\ell (T_{iik})=-3\lambda_i u_\ell u_k+\nu_i e_\ell (u_k),
	\label{humanlem4}\\\
	&U_{ijk\ell}=e_\ell (T_{ijk})
	-\sum_s\bigl(c_{si\ell}T_{sjk}+c_{sj\ell}T_{isk}
	+c_{sk\ell}T_{ijs}\bigr).
	\label{eq:construct-K}
\end{align}
	Using \eqref{eq:energy-for-symbol} and \eqref{human1}, together with \eqref{humanlem2}, \eqref{humanlem3}, \eqref{humanlem4}, and \eqref{eq:construct-K}, one can obtain
\begin{align}
		&0=\sum_r w_r u_r Q_{ir}+6T_{123} T_i
		+6\lambda_i\nu_i u_i^3+3S_0 u_i(u_j^2-u_k^2)
		-6\lambda_iT_{123} u_j u_k,
		\label{eq:first-energy-symbol}\\
		& T_i=\nu_iQ_{jk}-\rho_i(T_{123}u_i-\nu_i u_j u_k).
		\label{eq:offdiag-gauss-symbol}\\
			&1+\lambda_i\lambda_j
		=\frac{\nu_jQ_{ii}-\nu_iQ_{jj}}{\lambda_j-\lambda_i}
		+\frac{-3\lambda_j(\lambda_j-\lambda_i)-2\nu_j^2}
		{(\lambda_j-\lambda_i)^2}u_i^2\notag	\\
		&\qquad\qquad\quad+\frac{3\lambda_i(\lambda_j-\lambda_i)-2\nu_i^2}
		{(\lambda_j-\lambda_i)^2}u_j^2
		+\frac{2T_{123}^2}{(\lambda_i-\lambda_k)(\lambda_j-\lambda_k)},\label{eq:diag-gauss-symbol}
	\end{align}
	In \eqref{eq:first-energy-symbol} and
	\eqref{eq:offdiag-gauss-symbol}, $(i,j,k)$ is cyclic, whereas in
	\eqref{eq:diag-gauss-symbol} the indices $i,j,k$ are pairwise distinct. 
	
	Substitute \eqref{eq:general-cubic-characteristic} into
	\eqref{eq:first-energy-symbol} for $i=1,2$, and use
	\eqref{eq:offdiag-gauss-symbol} to eliminate $T_1$ and $T_2$.
	Multiply the equation with $i=2$ by $\varepsilon\sqrt t$ and add it
	to the equation with $i=1$. Since
	\[
	w_1=2R^2P,\quad w_2=2R^2Q,\quad w_3=2R^2K,
	\]
	and
	\(
	(\nu_1,\nu_2,\nu_3)=R(-t,-1,t+1),
	\)
	the terms containing $Q_{12},Q_{13},Q_{23}$ cancel, and the only
	remaining $Q$-term is
	\(
	2R^2Pv\,(Q_{11}-tQ_{22}).
	\)
	Equation \eqref{eq:diag-gauss-symbol} with
	$(i,j,k)=(1,2,3)$ expresses $Q_{11}-tQ_{22}$ in terms of
	$v,z$ and $T_{123}$. Substituting \eqref{eq:general-cubic-characteristic}
	once more and simplifying the remaining scalar terms by
	\eqref{eq:energy-for-symbol}, we obtain
	\begin{equation}\label{eq:positive-first-compatibility}
		0=vR^2\left\{-\frac{3Pt(t+1)}D
		-\frac{N_1(t)}{D^2Q^3}v^2
		-\frac{N_2(t)}{9D^2Q}z^2\right\}
	\end{equation}
where
	\begin{align*}
		N_1(t)&=2t^5(1+t)^4(420+1304q+1543q^2+822q^3+166q^4),\\
		N_2(t)&=2t^3(1+t)^3(134+467q+490q^2+166q^3)
	\end{align*}
	and, $q=t+t^{-1}\ge2$. 
	Every term in the bracket of \eqref{eq:positive-first-compatibility} is strictly negative or nonpositive, and the first is
	strictly negative. Therefore \(v=0\). 
	Equation \eqref{eq:energy-for-symbol} now gives \(z\ne0\).
	Changing the principal frame sign permits \(\varepsilon=1\). Thus
	\begin{equation}\label{eq:char-point-normal-form}
		u=(0,0,z),\quad T_{123}=\frac{KR}{3\sqrt t}z,
		\quad z^2=\frac{Dt(S_0-3)}{2K(t+1)^2}.
	\end{equation}
	
	Applying \eqref{eq:offdiag-gauss-symbol}-\eqref{eq:diag-gauss-symbol} repeatedly at
	\eqref{eq:char-point-normal-form} leaves exactly the two parameters
	\(k:=Q_{33}\) and \(y:=Q_{23}\). Explicitly,
	\begin{align}
		&Q_{11}=-\frac{1+\lambda_1\lambda_3+tk}{t+1}
		+\left[-\frac K{t+1}+\frac{2K^2}{9t^2(t+1)^2}\right]z^2,\\
		&Q_{22}=\frac{t(1+\lambda_2\lambda_3)-k}{t+1}
		+\left[\frac K{t(t+1)}-\frac{2K^2}{9(t+1)^2}\right]z^2,\\
		&Q_{12}
		=-\frac{\sqrt{t}\,k}{t+1}
		+(t-1)\left[
		\frac{(2t^2+3t+2)K}
		{3t\sqrt{t}(t+1)^2}
		-\frac{\sqrt{t}}{K}
		\right]z^2.\quad Q_{13}=\sqrt ty,
		\label{eq:char-hessian}\\
		&(T_1,T_2,T_3)=
		\left(-tRy,-\sqrt tRy,-\sqrt tR\left[k+\frac{t^2-1}{K}z^2\right]\right).
		\label{eq:char-T-derivative}
	\end{align}

	We next impose necessary identities for \(V=\nabla U\).
	By curvature components \eqref{found:Gauss},
the first three indices of \(V\) are symmetric, and its two other adjacent
	commutations are
	\begin{align}
		V_{ijk\ell m}-V_{ij\ell km}
		&=-\sum_s\bigl(
		R_{si k\ell;m}A_{sj}+ R_{si k\ell}T_{sjm}
		+R_{sj k\ell;m}A_{is}+ R_{sj k\ell}T_{ism}\bigr),\notag\\
		V_{ijk\ell m}-V_{ijkm\ell}
		&=-\sum_s\bigl(
		 R_{si\ell m}T_{sjk}
		+ R_{sj\ell m}T_{isk}
		+ R_{sk\ell m}T_{ijs}\bigr).
		\label{eq:fifth-order-commutations}
	\end{align}
	Consequently the \(243\) ordered components of \(V\) can be written as
	\(21\) components indexed by nondecreasing five-tuples plus explicitly known offsets.
	For each nondecreasing five-tuple choose its component as the unknown, and
	use adjacent transpositions and \eqref{eq:fifth-order-commutations} to
	express every other ordering. 
	The necessary linear equations in these \(21\) unknowns are
	\begin{align}
		0=&\sum_i V_{iik\ell m},\notag\\
			0=&\sum_i\lambda_i V_{iik\ell m}
		+\sum_{i,j}\bigl(T_{ijm}U_{ijk\ell}
		+U_{ijkm}T_{ij\ell}+T_{ijk}U_{ij\ell m}\bigr),\notag\\
		0=&\sum_{i,j,k}T_{ijk}V_{ijk\ell m}
		+\sum_{i,j,k}U_{ijk\ell}U_{ijkm}.
		\label{eq:linear-fifth-constraints}
	\end{align}
	The first two are imposed for all \((k,\ell,m)\in\{1,2,3\}^3\), and
	the third is imposed for \(\ell\le m\).
	They are respectively the differentiated minimality, the third derivative
	of \(|A|^2/2\), and the second derivative of \(|T|^2/2=S(S-3)/2\).
	
	Only nine equations in \eqref{eq:linear-fifth-constraints} are needed.
	Denote their right-hand sides, with the indicated normalizations, by
	\begin{align*}
		\mathcal E_{k\ell m}
		&=\sum_i V_{iik\ell m},\\
		\mathcal F_{k\ell m}
		&=\frac1R\left[\sum_i\lambda_i V_{iik\ell m}
		+\sum_{i,j}\bigl(T_{ijm}U_{ijk\ell}
		+U_{ijkm}T_{ij\ell}+T_{ijk}U_{ij\ell m}\bigr)\right],\\
		\mathcal H_{\ell m}
		&=\frac1{Rz}\left[\sum_{i,j,k}T_{ijk}V_{ijk\ell m}
		+\sum_{i,j,k}U_{ijk\ell}U_{ijkm}\right].
	\end{align*}
	The factors \(R\) and \(z\) are nonzero. Set
	\[
	\begin{split}
		\mathcal C:={}&-2t(1+2t^2)\mathcal E_{113}
		+3t(1+4t)\mathcal F_{113}+3t\mathcal H_{11}\\
		&-6\sqrt t(1+t^2)\mathcal E_{123}
		+9\sqrt t(t-1)\mathcal F_{123}+3\sqrt t(1+t)\mathcal H_{12}\\
		&-2(t^2+2)\mathcal E_{223}
		-3(t+4)\mathcal F_{223}+3t\mathcal H_{22}.
	\end{split}
	\]
	Since every term vanishes, we get \(\mathcal C=0\).
	
	To evaluate this combination $\mathcal{C}$, use
	\eqref{eq:fifth-order-commutations} to put the indices of each component
	of \(V\) in nondecreasing order. All coefficients of the resulting
	\(21\) ordered representatives cancel. This cancellation is performed
	with $k$ and $y$ unrestricted. The remaining terms involve only
	\(A,T,U\), already specified by \eqref{eq:char-point-normal-form},
	\eqref{eq:char-hessian}, \eqref{eq:char-T-derivative}, and
	\eqref{eq:construct-K}, and form a quadratic polynomial in the real
	parameters $k$ and $y$.
Define
	\begin{align*}
		P_0(q)&:=48+90q+53q^2+9q^3,\\
		P_1(q)&:=1112+5728q+10876q^2+8925q^3
		+2871q^4+266q^5+144q^6+48q^7,\\
		P_2(q)&:=5944+49534q+163723q^2+266862q^3
		+204829q^4+34408q^5\\
		&\qquad+q^6\left[(2368q+9216)(q^2-4)+3952q+966\right].
	\end{align*}
	For later use, set
	\[
	\mathcal N(q)
	=729(2q+1)^4P_0(q)
	+27(S_0-3)(2q+1)^2P_1(q)+(S_0-3)^2P_2(q),
	\]
	and
	\[
	\Lambda(q)=(q-2)(q+2)^2,\quad
	J(q)=22q^2+30q+11.
	\]
	Also put
	\[
	A_0(q)=96q^4+566q^3+950q^2+667q+178
	\]
	and
	\begin{equation}\label{eq:characteristic-square-center}
		\kappa
		=-\frac{t-1}{t+1}\,
		\frac{(S_0-3)A_0(q)+27(2q+1)^2(3q+2)}
		{144(q+1)(q+2)(2q+1)^2}.
	\end{equation}
	Finally, define
	\begin{equation}\label{eq:characteristic-remainder}
		\mathscr R(q)
		=8\mathcal N(q)
		-9\Lambda(q)
		\left[J(q)(S_0-3)+9(2q+1)^2\right]^2.
	\end{equation}
	After using \eqref{eq:fifth-order-commutations} to cancel all components of $V$ in $\mathcal C$,
	substitute \eqref{humanlem2}--\eqref{eq:construct-K} and \eqref{eq:char-point-normal-form}--\eqref{eq:char-T-derivative} into the remaining terms, collect the
	result as a quadratic polynomial in $k$ and $y$, and complete the square
	in $k$, this gives 
	\begin{equation}\label{eq:unrestricted-characteristic-certificate}
		\frac{z}{24Rt(t+1)^2}\,\mathcal C
		=(k-\kappa)^2+(t+1)y^2
		+\frac{\mathscr R(q)}
		{20736(q+2)^3(q+1)^2(2q+1)^4}.
	\end{equation}
	It remains to prove that the last term in
	\eqref{eq:unrestricted-characteristic-certificate} is positive. Expanding
	$\mathscr R$ only in $S_0-3$ gives
	\begin{align}	\label{eq:characteristic-remainder-expanded}
		\mathscr R(q)
		=&729(2q+1)^4\bigl(8P_0(q)-\Lambda(q)\bigr)
		+54(S_0-3)(2q+1)^2
		\bigl(4P_1(q)-3\Lambda(q)J(q)\bigr)\notag\\
		&+(S_0-3)^2\bigl(8P_2(q)-9\Lambda(q)J(q)^2\bigr).
	\end{align}
	For $q\geq2$,
	\[
	0\leq\Lambda(q)
	=q^3+2q^2-4q-8\leq2q^3
	\]
	and
	\[
	J(q)\leq
	\left(22+15+\frac{11}{4}\right)q^2
	=\frac{159}{4}q^2.
	\]
	The displayed formulas for $P_0,P_1,P_2$ give
	\[
	P_0(q)\geq9q^3,\quad
	P_1(q)\geq266q^5,\quad
	P_2(q)\geq3952q^7+966q^6.
	\]
	For the last estimate, use $q^2-4\geq0$ and discard the remaining
	nonnegative terms. Consequently,
	\begin{align*}
		8P_0-\Lambda
		&\geq72q^3-2q^3
		=70q^3>0,\\
		4P_1-3\Lambda J
		&\geq1064q^5
		-3(2q^3)\frac{159}{4}q^2
		=\frac{1651}{2}q^5>0,\\
		8P_2-9\Lambda J^2
		&\geq31616q^7+7728q^6
		-9(2q^3)\left(\frac{159}{4}q^2\right)^2\\
		&=\frac{25399}{8}q^7+7728q^6>0.
	\end{align*}
	It follows from \eqref{eq:characteristic-remainder-expanded} that
	\[
	\mathscr R(q)>0,\quad(q\geq2).
	\]
	
	Since $\mathcal C=0$, identity
	\eqref{eq:unrestricted-characteristic-certificate} now gives
	\[
	0=(k-\kappa)^2+(t+1)y^2
	+\frac{\mathscr R(q)}
	{20736(q+2)^3(q+1)^2(2q+1)^4},
	\]
	which is impossible: $k,y,\kappa$ are real, the first two terms are
	nonnegative, and the last term is strictly positive. Therefore a
	nonzero complex vector cannot satisfy
	\[
	q_0(\xi)=q_1(\xi)=q_2(\xi)=0.
	\]
	This proves the first part of Lemma \ref{lem:no-complex-characteristic}.

	Next, let
	\begin{equation*}
		(A_n,T_n,U_n,V_n)\longrightarrow
		(A_\infty,T_\infty,U_\infty,V_\infty)
	\end{equation*}
	be a convergent sequence of normalized tensor quadruples arising
	from pointed local minimal hypersurfaces with $|A_n|^2\equiv S_0>3$.
	By Lemma~\ref{lem:repeated}, $A_\infty$ has three distinct eigenvalues.
	After passing to a subsequence and making convergent orthogonal
	changes of basis, we may express all tensors in consistently ordered
	principal orthonormal bases, with
	\begin{equation*}
		\lambda_{i,n}\longrightarrow\lambda_{i,\infty},
		\quad
		\nu_{i,n}\longrightarrow\nu_{i,\infty}\neq0.
	\end{equation*}
	By $T_{iir}=\nu_i u_r$ and
	\eqref{humanlem2}--\eqref{eq:construct-K}, the quantities
	$u_r,Q_{ij},T_i$ depend continuously on the framed tensors $(A,T,U)$
	when the eigenvalue differences are nonzero. Thus $Q_{ij}=Q_{ji}$,
	\eqref{eq:energy-for-symbol}, and
	\eqref{eq:first-energy-symbol}--\eqref{eq:diag-gauss-symbol}
	pass to the limit.
	The identities \eqref{eq:fifth-order-commutations} and
	\eqref{eq:linear-fifth-constraints} also pass to the limit, with
	$R$ and $\nabla R$ expressed in terms of $A,T$ by
	\eqref{found:Gauss} and \eqref{humanxie}.
	
	Suppose, for contradiction, that some nonzero $\xi\in\mathbb C^3$
	satisfies
	\begin{equation*}
		q_0^\infty(\xi)=q_1^\infty(\xi)=q_2^\infty(\xi)=0.
	\end{equation*}
	This assumption gives \eqref{eq:general-cubic-characteristic}
	at the limit. However, no such condition is imposed on the approximating
	sequence. Using the identities established above, the same algebraic
	reductions give \eqref{eq:char-point-normal-form}, with
	$z_\infty\neq0$, and then
	\eqref{eq:unrestricted-characteristic-certificate} for the limiting
	tensors. Since the limiting equations
	\eqref{eq:linear-fifth-constraints} give $\mathcal C_\infty=0$,
	we obtain
	\begin{equation*}
		0=(k_\infty-\kappa_\infty)^2
		+(t_\infty+1)y_\infty^2
		+\frac{\mathscr R(q_\infty)}
		{20736(q_\infty+2)^3(q_\infty+1)^2(2q_\infty+1)^4}.
	\end{equation*}
	Here $t_\infty>0$ and
	$q_\infty=t_\infty+t_\infty^{-1}\ge2$.
	By \eqref{eq:characteristic-remainder-expanded} and the estimates
	following it, $\mathscr R(q_\infty)>0$.
	The right-hand side is therefore strictly positive, which is a contradiction.
	This proves the assertion for the limiting polynomials.
	\end{proof}
\begin{remark}
	Notice that we do not obtain the limiting conclusion merely by taking
	limits of the statements for \(q_0^n,q_1^n,q_2^n\).  Actually, the additional tensors \(U_\infty,V_\infty\)
	allow all finite tensor identities used in the contradiction argument
	to pass to the limit.
\end{remark}

\section{Compactness of the space of local minimal hypersurfaces}\label{sec3}
\label{sec:compactness}
The main issue in this section is a realization problem: whether a
limit $(A_\infty, T_\infty, U_\infty)$ of normalized tensor triples $(A_n,T_n,U_n)$ arising
from local pointed minimal hypersurfaces in $S^4$ is itself realized at a point of some local
minimal hypersurface in $S^4$. To address this, we represent the hypersurfaces as normalized
graphs, converting the geometric conditions into analytic differential
equations for the graph functions. Lemma \ref{lem:no-complex-characteristic} yields a
system of partial differential equations determining fifth derivatives of the graph function from derivatives of order at most
four. An associated ordinary differential equation constructs candidate
graph functions on a common neighborhood. Then analytic continuation and
passing to the limit establish the required realization. Together
with uniform derivative bounds of $A_n$, this proves compactness of the set of
normalized fourth-order jets of graph functions.
\subsection{Normalized graphs}
\label{subsec:graphs}
Let $M^3$ be a local hypersurface in $S^4$.
Choose a point $p\in M^3$, a unit normal $\nu$, and an orthonormal basis of $T_pM$. Let $\{E_1,E_2,E_3,E_4,E_5\}$ be the standard orthonormal basis of $\mathbb{R}^5$.
After an ambient orthogonal transformation, we may assume that 
$p=E_5, \nu=E_4,$
and that the chosen tangent frame is $(E_1, E_2, E_3)$. Projection onto the first three coordinates
has invertible differential on the tangent space at $p$. The inverse
function theorem therefore gives a unique graph on a sufficiently
small neighborhood of $0\in\mathbb R^3$:
\begin{equation}\label{eq:normalized-graph-new}
 X(x)=\left(x_1,x_2,x_3,u(x),\sqrt{1-|x|^2-u(x)^2}\right),
 \quad u(0)=0,\quad Du(0)=0.
\end{equation}
Here the uniqueness is up to the chosen orthnormal basis of $T_pM$ and unit normal $\nu$.

Here we remark that the graph function $u$ in \eqref{eq:normalized-graph-new} is in fact real analytic if $M^3$ is minimal. Indeed, the round metric
on $S^4$ is real analytic and the hypersurface has zero mean curvature,
so this follows from the analytic regularity of constant mean curvature
hypersurfaces \cite[Lemma~4.16]{AMM2013}.

For a smooth function $u$ defined near a point $x\in\mathbb R^3$, its
$m$-jet at $x$, denoted by $j_x^m u$, is 
\[
j_x^m u
=(D^\alpha u(x))_{|\alpha|\le m}.
\]
Thus $j_x^m u$ records the Taylor coefficients of $u$ at $x$ up to
order $m$. We also write $j^m u$ for the corresponding jet field
$x\mapsto j_x^m u$. For the normalized graphs considered below, the zeroth-order and first-order
parts of $u$ are fixed by $u(0)=0$ and $Du(0)=0$.

The formula of the second fundamental form and covariant differentiation give
a triangular relation at $p$:
\begin{equation}\label{eq:triangular-new}
 \Sym(\nabla^mA)(p)
 =D^{m+2}u(0)+\mathcal R_m(j^{m+1}_0u),
 \qquad 0\le m\le3.
\end{equation}
Here $\Sym(\nabla^mA)$ denotes the symmetrization of $\nabla^mA$, and the term $\mathcal R_m$ depends only
lower derivatives of $u$. Normal convention is chosen such that $A(p)=D^2u(0)$. To see the leading term of \eqref{eq:triangular-new}, differentiate the second
fundamental form $m$ times: only differentiation of the second
derivatives of $u$ can produce $D^{m+2}u$, and differentiating its
coefficients or the connection produces lower-order terms. Conversely, the full covariant derivatives of the second fundamental form through order m are universal expressions in \(j_0^{m+2}u\).

Throughout this section, a local solution  means a germ of an oriented
smooth minimal hypersurface in $S^4$ with constant $|A|^2=S_0>3$. 
Consequently, by \eqref{eq:triangular-new}, $j^4_0u$ and the tensor triple
$(A,T,U)|_p$ determine each other mutually. Moreover, for a sequence of local solutions $(M_n^3, p_n)$ with graph functions $u_n$, then $j_0^4u_n$ converges if and only if $(A_n,T_n,U_n)|_{p_n}$ converges.
Actually, we consider $\{j_0^4u_n\}$ as a sequence of vectors in $\mathbb{R}^{35}$.
Here convergence of $j_0^4u_n$ is understood in the usual Euclidean
topology of $\mathbb R^{35}$.
Simliarly, the same statement holds for $j^5_0u$ and
$(A,T,U,V)|_p$. 
Define
\begin{equation}\label{eq:jet-space-new}
 \mathscr G_{S_0}
: =\{j^4_0u:\ u\text{ is a graph function of a local solution}\},
\end{equation}
where we consider $\mathscr G_{S_0}$ as a topological subspace of $\mathbb{R}^{35}$.
In subsection \ref{subsec:radial-ode-new}, we will prove $ \mathscr G_{S_0}$ is compact.

\subsection{Uniform bounds for the covariant derivatives of $A$}
\label{subsec:bounds-new}
Let $M^3$ be a minimal hypersurface in $S^4$ with the shape operator $A$ of a fixed unit normal vector field. Assume $|A|^2=S_0$ is constant.
 Write $J_m=\nabla^mA$.
There are universal contraction polynomials $P_m, Q_m$, with coefficients
depending on $S_0$, such that
\begin{equation}\label{eq:polynomial-identities-new}
 \Delta J_m=Q_m(J_0,\ldots,J_m),\quad
 |J_m|^2=P_m(J_0,\ldots,J_{m-1})\quad(m\ge1).
\end{equation}
We will prove the first identity of \eqref{eq:polynomial-identities-new} by induction. When $m=0$, it is valid by \eqref{found:Simons}, and assume
\[
\Delta J_m=Q_m(J_0,\ldots,J_m),
\]
then
\[
\begin{aligned}
	\Delta J_{m+1}
	&=
	\nabla(\Delta J_m)+[\Delta,\nabla]J_m\\
	&=\nabla(\Delta J_m)+R*J_{m+1}+(\nabla R)*J_m,
\end{aligned}
\]
where $*$ denotes a universal contraction.  Hence
\[
\Delta J_{m+1}
=:
Q_{m+1}(J_0,\ldots,J_{m+1}).
\]
For the second identity of \eqref{eq:polynomial-identities-new}, we will also prove it by induction. When $m=1$, it follows from \eqref{found:Simons}.
Assume
\[
|J_m|^2=P_m(J_0,\ldots,J_{m-1}).
\]
Then the Bochner-type identity yields
\[
|J_{m+1}|^2
=
\frac12\Delta |J_m|^2-\langle J_m,\Delta J_m\rangle
=:P_{m+1}(J_0,\ldots,J_m),
\]
where the first identity in \eqref{eq:polynomial-identities-new} is used
in the second equality.

Since $|J_0|^2=S_0$, by induction on $m$, and using the second identity in \eqref{eq:polynomial-identities-new}, we
shows that each $J_m$ is uniformly bounded once the lower-order
$J_k$ are bounded.
Therefore, there exists constants $C_m(S_0)$ such that
\begin{equation}\label{eq:finite-bounds-new}
 |\nabla^mA|\le C_m(S_0),\quad 0\le m\le3,
\end{equation}
at every point of $M^3$. By \eqref{eq:triangular-new} and \eqref{eq:finite-bounds-new} for a graph function $u$ of $M$ near $p$, we have $|j_0^5u|\leq C'_5(S_0)$ and also $|j_0^4u|\leq C'_4(S_0)$, where $ C'_5(S_0)$ and  $C'_4(S_0)$ are constants
dependent only on $S_0$.
\subsection{A determined system for fourth-order jets}
\label{subsec:fourth-order-new}
Let $M^3$ be a local minimal hypersurface in $S^4$ with three distinct principal curvatures and contant $|A|^2=S_0>3$. Choose $p\in M^3$, and suppose $u$ is the graph function of $M^3$ near $p$.
Let $A_u$ denote the shape operator of the normalized graph determined by $u$, and
\(
A_u(p)=D^2u(0).
\)
In graph coordinates consider the analytic equations
\begin{align*}
	H(u)=\tr A_u=0,\quad
	G(u)=|A_u|^2-S_0=0,\quad
	E_1(u)=|\nabla A_u|^2-S_0(S_0-3)=0.
\end{align*}
Their orders in $u$ are $2,2,3$ respectively. At $p$, their
linearized principal symbols \cite[Chapter~2, Section~9]{TaylorPDE1}, up to nonzero scalar factors, are
\begin{equation}\label{eq:symbols-new}
	q_0(\xi)=\sum_i\xi_i^2,\quad q_1(\xi)=A_p(\xi,\xi),\quad
	q_2(\xi)=T_p(\xi,\xi,\xi).
\end{equation}
\begin{proposition}\label{spacedec}
	Let 
$\mathcal P_d$
	denote the space of homogeneous complex 
	polynomials defined on $\mathbb{C}^3$ of degree $d$. Then $\mathcal P_5$ is the sum of the following three subspaces:
	\begin{equation}\label{eq:degree-five-span-new}
		\mathcal P_5=q_0\mathcal P_3+q_1\mathcal P_3
		+q_2\mathcal P_2.
	\end{equation}
\end{proposition}
\begin{proof}
At the point $p$, choose a principal orthonormal basis such that $A_p$ is diagonal, with distinct principal curvatures $\lambda_i$, and put
$\delta_i=\lambda_j-\lambda_k$ cyclically. The equations $q_0(\xi)=q_1(\xi)=0$
imply that $(\xi_1^2,\xi_2^2,\xi_3^2)$ is proportional to
$(\delta_1,\delta_2,\delta_3)$. Each $\delta_i$ is nonzero. Choose square roots $\sqrt{\delta_i}\in\mathbb C$ and the
invertible complex change of variables
\(\xi_i=\sqrt{\delta_i}\,y_i\).
  Then
\[
q_0(y)=\sum_i\delta_i y_i^2,
\quad
q_1(y)=\sum_i\lambda_i\delta_i y_i^2.
\]
Since $\sum_i\delta_i=0$, $\sum_i\lambda_i\delta_i=0,$
and the two vectors \((\delta_1,\delta_2,\delta_3)\), \((\lambda_1\delta_1,\lambda_2\delta_2,\lambda_3\delta_3)\) are linearly independent, these two vectors span 
\[
\{(a_1,a_2,a_3)\in\mathbb C^3:a_1+a_2+a_3=0\}.
\]
Hence, $y_1^2-y_2^2, y_1^2-y_3^2
\in\operatorname{span}_\mathbb{C}\{q_0,q_1\}.$
Set
\[
\overline{\mathcal P}_2
:=
\mathcal P_2/\operatorname{span}_\mathbb{C}\{q_0,q_1\},\quad \overline{\mathcal P}_5
:=
\mathcal P_5/(q_0\mathcal P_3+q_1\mathcal P_3).
\]
Write $[f]$ for the equivalence class of $f$ in $\mathcal P_2$ or $\mathcal P_5$. Therefore we have
$
[y_1^2]=[y_2^2]=[y_3^2]
$ in $\overline{\mathcal P}_2$.
Similarly,
multiplying the preceding quadratic relations by cubic monomials shows
that every class in $\overline{\mathcal P}_5$ is a linear combination of
\begin{equation}\label{P5}
	[y_1^5], \quad[y_1^4y_2], \quad[y_1^4y_3], \quad[y_1^3y_2y_3].
\end{equation}
For
\(
v_{\varepsilon\eta}:=(1,\varepsilon,\eta),
\varepsilon,\eta\in\{1,-1\},\)
we define evaluation map:
\[
g:
\overline{\mathcal P}_5\longrightarrow \mathbb C^4,
\quad
g([f])
=
\bigl(f(v_{\varepsilon\eta})\bigr)_{\varepsilon,\eta\in\{\pm1\}}.
\]
This map $g$ is well defined since $q_0(v_{\varepsilon\eta})=q_1(v_{\varepsilon\eta})=0$. At the four points
$v_{\varepsilon\eta}$, the four 
representatives \eqref{P5} have the values $1, \varepsilon,\eta, \varepsilon\eta,$
respectively.
Hence, if
\[
f(v_{\varepsilon\eta})
=
a+b\varepsilon+c\eta+d\varepsilon\eta
\]
vanishes for all four choices of $(\varepsilon,\eta)$, then $a=b=c=d=0.$
Thus the evaluation map $g$ is injective, and hence an isomorphism since
both spaces have dimension four. Simliarly, one can obtain an evaluation isomorphism between $\overline{\mathcal P}_2$ and $\mathbb C^4$.


Lemma~\ref{lem:no-complex-characteristic} says that $q_2$ is nonzero at
each of $v_{\epsilon\eta}$. For any $f\in\mathcal P_5$, the evaluation isomorphism between $\overline{\mathcal{P}}_2$ and $\mathbb{C}^4$ implies that there exists $p_2\in\mathcal P_2$ satisfying
\[
p_2(v_{\varepsilon\eta})
=\frac{f(v_{\varepsilon\eta})}{q_2(v_{\varepsilon\eta})},
\quad\varepsilon,\eta\in\{1,-1\}.
\]
Then $f-q_2p_2$ vanishes at all four points $v_{\varepsilon\eta}$, which implies that the equivalent
class $[f-q_2p_2]=0$ in $\overline{\mathcal{P}}_5$ . This proves
\eqref{eq:degree-five-span-new}.
\end{proof}

Now we estabilsh the differential equations of fourth-order jet $j^4_0u$. Differentiate $H$, $G$, and $E_1$ respectively $3$, $3$, and $2$
times. The resulting equations have the form
\begin{equation}\label{human}
	M(x,j_x^4u)
	\bigl(D^\alpha u(x)\bigr)_{|\alpha|=5}
	=
	b(x,j_x^4u),
\end{equation}
where $M(x,j_x^4u)$ is a $26\times21$ matrix and $b(x,j_x^4u)$
depends only on derivatives of $u$ of order at most four. At $x=0$,
\eqref{eq:degree-five-span-new} implies exactly that the coefficient
matrix $M(0, j^4_0u)$ has rank $21$ over $\mathbb C$, hence also over $\mathbb R$.
Therefore we can choose a nonsingular $21\times21$ submatrix $M_0(0, j^4_0u)$,
  it remains nonsingular in a neighborhood $\mathcal{U}$ of $(0, j^4_0u)$ in $\mathbb{R}^3\times \mathbb{R}^{35}$. Consequently, for each $\alpha=(\alpha_1,\alpha_2,\alpha_3)$ with $|\alpha|=5$, there exists an real analytic function $\Phi_\alpha$ near $(0, j^4_0u)$ such that
\begin{equation}\label{eq:fifth-determined-new}
	D^\alpha u=\Phi_\alpha(x,j_x^4u).
\end{equation}
 Now set
\[
Y:=j_x^4u=\bigl(D^\alpha u\bigr)_{|\alpha|\le4}.
\]
By \eqref{eq:fifth-determined-new}, there exists three real analytic $\mathbb{R}^{35}$-valued functions $\mathcal F_1, \mathcal F_2,\mathcal F_3$ near $(0, j^4_0u)$ satifying
\begin{equation}\label{eq:complete-system-new}
	\partial_iY=\mathcal F_i(x,Y),\quad i=1,2,3.
\end{equation}
This is a system of real analytic first-order equations for the $35$
components of $Y$. 

\subsection{Compactness of $\mathscr G_{S_0}$}
\label{subsec:radial-ode-new}

\begin{theorem}\label{prop:compact-jets-new}
For fixed $S_0>3$, the set $\mathscr G_{S_0}$ is compact.
\end{theorem}
\begin{proof} 
The boundedness of $\mathscr G_{S_0}\subset\mathbb{R}^{35}$ follows from \eqref{eq:finite-bounds-new}. Next we will show that $\mathscr G_{S_0}$ is closed in $\mathbb{R}^{35}$. Choose a sequence of normalized graphs $u_n$ of local solutions with
$|A|^2=S_0$, and suppose $a_n=j^4_0u_n\to a_\infty$. Since the fifth-order jets of $u_n$ are
bounded by \eqref{eq:finite-bounds-new},  after taking some subsequence if necessary, we can assume $j^5_0u_n$ converges
to a lift $\widehat a_\infty$ of $a_\infty$. Let $(A_n,T_n,U_n,V_n)$ be the normalized tensor quadruples arising from the normalized graphs of $u_n$ at $x=0$, and $(A_\infty, T_\infty, U_\infty, V_\infty)$ corresponds to the lift $\widehat a_\infty$.

Applying Lemma~\ref{lem:repeated} and
\ref{lem:no-complex-characteristic} to the limiting normalizd tensor quadruples $(A_\infty, T_\infty, U_\infty, V_\infty)$,
it follows from the argument for \eqref{eq:fifth-determined-new} that
$\operatorname{rank}M(0,a_\infty)=21$.
Choose a \(21\times21\) submatrix \(M_0(0,a_\infty)\) of \(M(0,a_\infty)\) with nonzero determinant, and let $b_0$ consist of the
corresponding entries of $b$. On a neighborhood $\mathcal U$
of $(0,a_\infty)$ where $\det M_0\neq0$, define
\begin{equation*}
	\Phi(x,Y):=M_0(x,Y)^{-1}b_0(x,Y).
\end{equation*}
It is clear that $\Phi(x,Y)$ is real analytic near $(0, a_\infty)$. 
Thus simliar to \eqref{eq:complete-system-new}, there exist three real analytic $\mathbb{R}^{35}$-valued functions $\mathcal G_1, \mathcal G_2,\mathcal G_3$ near $(0, a_\infty)$ and we can define the system
\begin{equation}\label{humanP}
	\partial_iY=\mathcal G_i(x,Y),\quad i=1,2,3.
\end{equation}
In fact, for all sufficiently
large $n$, the jet field $Y_n(x)=j_x^4u_n$ satisfies \eqref{humanP}
near $x=0$, and $Y_n(0)=j^4_0u_n=a_n$.
Consider $Z(t)$, which satisfies the following ordinary differential equation: 
for an initial value $a$ near $a_\infty$ and a small parameter
$x\in\mathbb R^3$, 
\begin{equation}\label{eq:radial-ode-new}
 \frac{dZ}{dt}=\sum_{i=1}^3x_i\mathcal G_i(tx,Z),
 \quad Z(0)=a,\quad 0\le t\le1.
\end{equation}
There exist $r,\rho>0$ such that, for every parameter
$x\in B_r(0)\subset\mathbb R^3$ and every initial value
$a\in B_\rho(a_\infty)\subset\mathbb R^{35}$, the solution
$Z(t;x,a)$ of \eqref{eq:radial-ode-new} exists for all $t\in[0,1]$ and
depends real-analytically on $(x,a)$. 
Indeed, choose $r_0,\rho>0$ such that
\(	\overline{B_{r_0}(0)}
\times\overline{B_{2\rho}(a_\infty)}
\subset\mathcal U.\) Write $$C^*=\sup\{|\mathcal{G}_i(y,Y)|: i=1,2,3,(y, Y)\in \overline{B_{r_0}(0)}\times \overline{B_{2\rho}(a_\infty)}\}, $$ 
Then one can obtain that
$|Z'(t)|\le C_*|x|$ while $(tx,Z(t))$ remains this compact subset of $\mathcal{U}$, where $C_*=\sqrt{3}
C^*$. Choose $0<r<r_0$ with $C_*r<\rho/2$. For $|x|<r$ and
$|a-a_\infty|<\rho$, as long as the local solution exists and
remains in $B_{2\rho}(a_\infty)$, we have
\begin{equation*}
	|Z(t)-a_\infty|
	\le |a-a_\infty|+C_*|x|t
	<\frac{3\rho}{2},
	\quad 0\le t<1.
\end{equation*}
Hence, as long as the solution $Z(t)$ exists for $0\le t<1$, the trajectory
$(tx,Z(t))$ remains in the interior of
\(
\overline{B_{r_0}(0)}\times\overline{B_{2\rho}(a_\infty)}
\subset\mathcal U.
\)
In particular, it cannot reach the boundary of this compact set before
time $1$.
 If its maximal right endpoint were $T\le1$, boundedness
of $Z'(t)$ would give a limit $\lim_{t\nearrow T}Z(t)$ inside the ball. 
Local existence
with this limiting value at time $T$ would then extend the solution $Z(t)$
past $T$, which is a contradiction. Hence the solution $Z(t)$ exists on $[0,1]$ (See \cite[Theorems~1.2 and~1.4]{Chicone1999}). Moreover, since $\mathcal{G}_1, \mathcal{G}_2,\mathcal{G}_3$ are real analytic on $\mathcal{U}$. Therefore,
the local solution $Z(t;x,a)$ depends real-analytically on the
parameter $x$ and the initial value $a$ by
\cite[Theorem~1.3]{Chicone1999}. In particular, $Z(1;x,a)$ depends real-analytically on $(x,a)$.

Let $v_a(x)$ be the component of $Z(1;x,a)$ corresponding to the
zeroth derivative of $u$. The function $v_a(x)$ provides an analytic candidate function
on the ball $B_r(0)$. For $a=a_n$, the actual jet field
$j_{tx}^4u_n$ satisfies~\eqref{eq:radial-ode-new} for all
sufficiently small $x$. The uniqueness theorem of the ODE system implies
$ v_{a_n}(x)=u_n(x)$ on some neighborhood of $0$. If necessary, shrink $B_r(0)$ so that all functions $v_{a_n}$ give valid normalized graphs (3.1) on $B_r(0)$. Consequently,
$H(v_{a_n})$ and $G(v_{a_n})$ are real analytic on the ball
$B_r(0)$ and vanish near $0$. By the real analyticity,
they vanish throughout $B_r(0)$.

Since $v_a(x)$ depends real analytically on $x$ and the initial value
$a$, the convergence $a_n\to a_\infty$ implies
\begin{equation*}
	\sup_{x\in\overline{B_{r'}(0)}}
	\left|D_x^\alpha v_{a_n}(x)-D_x^\alpha v_{a_\infty}(x)\right|
	\xrightarrow{n\rightarrow \infty } 0
	\quad (0<r'<r)
\end{equation*}
for every multi-index $\alpha$. Passing
to the limit shows
\[
 H(v_{a_\infty})=G(v_{a_\infty})=0,
 \quad
 j^4_0v_{a_\infty}=\lim_{n\rightarrow\infty} j^4_0v_{a_n}=a_\infty.
\]
Hence $a_\infty$ is realized by a normalized graph defined by $v_{a_\infty}$ at $x=0$ and
belongs to $\mathscr G_{S_0}$. 
Cosequently, $\mathscr G_{S_0}$ is closed in $\mathcal{R}^{35}$.

Finally, since $\mathscr G_{S_0}$ is bounded and closed in $\mathbb{R}^{35}$, it follows that $\mathscr G_{S_0}$ is compact.
\end{proof}


\section{Proof of the rigidity theorem}\label{sec:finish}

\subsection{Proof of Theorem~\ref{thm:main}}
\begin{proof}
	Fix $S_0>3$, and suppose that $\mathscr G_{S_0}\neq\varnothing$.
	By Theorem~\ref{prop:compact-jets-new},
	$\mathscr G_{S_0}$ is compact.
	First, 
	for the compact set $\mathscr G_{S_0}$, we
	define 
	\[
	F:\mathscr G_{S_0}\longrightarrow\mathbb R,
	\quad F(j^4_0u)=\tr(A^3_u(0)).
	\]
	It is clear that $F$ is continuous. For a local
	solution $N^3\subset S^4$ with a fixed local unit normal, write $J_N(p)$ for the normalized
	fourth-order jet of the graph function at $p$. Thus
	\begin{equation}\label{eq:pullback}
		f_N(p)=F(J_N(p)),
	\end{equation}
	where \(f_N=\tr(A_N^3)\).
	Let $b=\max_{\mathscr G_{S_0}} F$, and the compactness of $\mathscr G_{S_0}$ ensure that we can choose a local
	solution $N_{\max}^3\subset S^4$, a point $p\in N_{\max}$, and a normalized graph
	function $u$ near $p$ such that
$
	F(j_0^4u)=b.
$
By Lemma~\ref{lem:repeated}, the principal curvatures of
$N_{\max}$ are distinct at every point.
	At any $q\in N_{\max}$ near $p$ gives $J_{N_{\max}}(q)\in\mathscr G_{S_0}$,
	so
	\[
	f_{N_{\max}}(q)=F(J_{N_{\max}}(q))\le b=f_{N_{\max}}(p).
	\]
	Thus $df_{N_{\max}}(p)=0$. It follows from Lemma \ref{lem:critical}
	that $S_0=6$ and $b=0$. Similarly, $a=\min_{\mathscr G_{S_0}} F=0$. Consequently,
	\begin{equation}\label{eq:universal-zero}
		\mathscr G_{S_0}\ne\varnothing
		\quad\Longrightarrow\quad S_0=6,\quad F\equiv0.
	\end{equation}
 
Now, let $M$ be a hypersuface as in the assumption of Theorem \ref{thm:main}. At each point $p$, one can define a normalized graph near $p$. Consequently, the graph function
determines an element $J_M(p)$ in $\mathscr G_{S_0}$. Hence \eqref{eq:pullback} and
\eqref{eq:universal-zero} imply $S_0=6$ and
$f(p)=\tr(A_M^3(p))=F(J_M(p))=0$. Equation~\eqref{eq:main-cubic} becomes
\[
 \det(t\Id-A_M(p))=t^3-3t=t(t^2-3).
\]
At each point  $p$, $M$ has constant principal curvatures $\sqrt3,0,-\sqrt3$. Therefore, $M$ is an isoparametric hypersurface  with three distinct principal curvatures in $S^4$. According to \cite{Cartan1939}, the proof is complete.
\end{proof}
\subsection{Proof of Corollary~\ref{cor:full-local}}
\begin{proof}
 By the Gauss equation, $S=|A|^2$ is constant. Simons’ identity therefore gives
    \[
    0=\frac12\Delta S=|\nabla A|^2+S(3-S),
   \]
    and hence $S=0$ or $S\ge3$.
    
    If \(S=0\), then \(A=0\), so \(M^3\) is totally geodesic. If \(S=3\), the same identity gives \(\nabla A=0\); consequently, the principal curvatures are constant.
    Suppose now that \(S>3\). By Lemma \ref{lem:repeated}, the principal curvatures are distinct at every point. Theorem \ref{thm:main} then implies that \(M^3\) is locally congruent to the minimal Cartan isoparametric hypersurface, and in particular \(S=6\).
    
    Thus \(S\in\{0,3,6\}\), and $M^3$ is locally congruent to a minimal isoparametric hypersurface in $S^4$. 
\end{proof}
\section*{Acknowledgements}
 
 The authors would like to thank Professor Zizhou Tang for helpful
 discussions and valuable comments.
 
 \section*{AI Disclosure}
 
 The AI tool ChatGPT (OpenAI) was used in the preparation of this
 manuscript. The authors have independently verified all AI-assisted
 content and take full responsibility for the final manuscript.




\end{document}